\documentclass{amsart}

\usepackage{amsmath, amssymb,amsfonts, amscd, xr}
\usepackage{glossaries}
\usepackage{imakeidx}
\usepackage[shortlabels]{enumitem}
\usepackage{hyperref}
\input{olddefs.tex}

\numberwithin{claim}{subsection}

\def\TO#1{^{(#1)}}

\newcounter{my_enumerate_counter}
\newcommand{\pushcounter}{\setcounter{my_enumerate_counter}{\value{enumi}}}
\newcommand{\popcounter}{\setcounter{enumi}{\value{my_enumerate_counter}}}

\title{Dependence of functions on their variables}

\author{Ilijas Farah}
\address{Department of Mathematics and Statistics\\
	York University\\
	4700 Keele Street\\
	North York, Ontario\\ Canada, M3J 1P3\\
	and 
	Ma\-te\-ma\-ti\-\v cki Institut SANU\\
	Kneza Mihaila 36\\
	11\,000 Beograd, p.p. 367\\
	Serbia}
\email{ifarah@yorku.ca}
\urladdr{https://ifarah.mathstats.yorku.ca}
\thanks{Partially supported by NSERC}
\date{\today}

\begin{document}
\maketitle

In this note we present a readable version of the proof of the main result of \cite{Fa:Dependence}, giving a sufficient and necessary condition for a function on a combinatorial cube to essentially (locally) depend on at most one variable (see the end of the paper for the motivation),  as well as some limiting results.  It will not be submitted for publication, but it will be included in the upcoming revised version of \cite{Fa:AQ}. 

In the following we identify $d\in \bbN$ with the set $\{0,\dots, d-1\}$ and $\pi_j$ stands for the projection from $\prod_{i<d} X_i$ onto the $j$-th coordinate,  $\pi_j(x_0,x_1,\dots x_{d-1})=x_j$. 
For a partition $d=u\sqcup v$, $x\in \prod_{i\in u} X_i$ and $y\in \prod_{i\in v} X_i$ we write $x^\frown y$ for $z\in X^d$ such that $z(i)=x(i)$ for $i\in u$ and $z(i)=y(i)$ for $i\in v$. 

\begin{thm}
	\label{T.Dependence} For all $d\geq 1$, sets $X_i$, for $i<d$ and $Y$, every $f\colon \prod_{i<d} X_i\to Y$ satisfies  exactly one of the following. 
	\begin{enumerate}
		\item \label{1.Dependence} There are $k\in \bbN$ and  partitions $X_i=\bigsqcup_{j<k} U_{i,j}$ such that for every $s\in k^d$ for some $j(s)<d$  and $g_s\colon U_{i,j(s)}\to Y$ the functions $f$ and $g_s\circ \pi_{j(s)}$ agree on $\prod_{i<d} U_{i,s(i)}$. 
		\item \label{2.Dependence} There are a partition $d=u\sqcup v$ into nonempty sets and $x_m\in \prod_{i\in u} X_i$ and $y_m\in \prod_{i\in v} X_i$, for $m\in \bbN$, such that for all $l$ and all $m<n$ we have 
		\[
		f(x_l{}^\frown y_l)\neq f(x_m{}^\frown y_n). 
		\]
		\pushcounter
	\end{enumerate}
\end{thm}

The proof of Theorem~\ref{T.Dependence} below uses compactness of the \v Cech--Stone compactification $\beta X$ of $X$ equipped with the discrete topology, a consequence of the Axiom of Choice not provable in ZF alone. In \S\ref{S.Choice} we show that (assuming that ZF has a model) in some models of ZF  its conclusion fails, and also show in ZF that the conclusion of  Theorem~\ref{T.Dependence} holds for every well-orderable set $X$.  The latter improves an observation due to the referee of \cite{Fa:Dependence}, where a proof in the case $X=\bbN$ had been sketched, while the former is new.

\section{Proof of Theorem~\ref{T.Dependence}}
We will prove the case of Theorem~\ref{T.Dependence} when the sets $X_i$ are equal for $i<d$. 
The case when all $X_i$ are of the same cardinality follows immediately, and the general case can be proven by appropriately changing the notation in the provided proof.   
The proof of Theorem~\ref{T.Dependence} is given after a few lemmas. 

\begin{lemma}
	\label{L.Dependence.1} For all $f\colon X^d\to Y$, possibilities \eqref{1.Dependence} and \eqref{2.Dependence} from Theorem~\ref{T.Dependence} exclude each other. 
\end{lemma}

\begin{proof} 
	Otherwise, we can  fix a partition $X=\bigsqcup_{j<k} U_j$, $j(s)$, and $g_s\colon U_{j(s)}\to Y$ such that $f$ and $g_s\circ \pi_{j(s)}$ agree on $\prod_{i<d} U_{s(i)}$ for all $s\in d^k$ as in \eqref{1.Dependence}. Also fix $u,v$, $x_m$, and $y_m$ as in \eqref{2.Dependence}. Let $s_u\in k^u$ be such that the set $Z=\{m\mid x_m\in \prod_{i\in u} U_{s_u(i)}\}$ is infinite, and let $s_v\in k^v$ be such that the set $Z'=\{m\in Z\mid y_m\in \prod_{i\in v} U_{s_v(i)}\}$ is infinite. Let $s=s_u^\frown s_v$ .  Fix $m<n$ in $Z'$. If $j(s)\in u$, then $f(x_m{}^\frown y_m)\neq f(x_m{}^\frown y_n)$ although $\pi_{j(s)}(x_m{}^\frown y_m)=\pi_{j(s)}(x_m{}^\frown y_n)$, contradicting the assumption that on $\prod_{i<d} U_{s(i)}$ the function $f$ depends only on the $j(s)$-coordinate. 
	Therefore $j(s)\in v$. Then $f(x_m{}^\frown y_n)\neq f(x_n{}^\frown y_n)$ and $\pi_{j(s)}(x_m{}^\frown y_m)=\pi_{j(s)}(x_m{}^\frown y_n)$, contradicting the assumption that  $f$ depends only on the $j(s)$-coordinate on $\prod_{i<d} U_{s(i)}$. 
\end{proof}

\begin{lemma}
	\label{L.Dependence.2} Assume that for $f\colon X^d\to Y$ there are a partition $d=u\sqcup v$ into nonempty sets and $x_m\in X^u$ and $y_m\in X^v$, for $m\in \bbN$,  that satisfy one of the following. 
	\begin{enumerate}
		\popcounter
		\item \label{5.Dependence} For all $l<m<n$ we have 
		\[
		f(x_l{}^\frown y_l)\neq f(x_l{}^\frown y_m)\quad \text{and}\quad  f(x_l{}^\frown y_m)=f(x_l{}^\frown y_n). 
		\]
		\item \label{4.Dependence}For all $m<n$ we have 
		\[
		f(x_m,y_m)\neq f(x_m,y_n)\quad \text{ and}\quad  f(x_m,y_m)\neq f(x_n,y_m).
		\]			  
		\pushcounter
	\end{enumerate}
	Then \eqref{2.Dependence} of Theorem~\ref{T.Dependence} applies. 
\end{lemma}

A standard bit of notation will come handy in the proof of Lemma~\ref{L.Dependence.2}. 
$\{l,m,n\}_<$ to denote the set $\{l,m,n\}$ while asserting that $l<m<n$ (analogous notation applies for any $d\geq 2$ in place of the number $3$.)

\begin{proof}Readers familiar with the canonical Erd\" os--Rado extension of Ramsey's theorem (\cite{erdos1950combinatorial}) surely don't need to read this proof  and processing a `slow' proof may take less effort on reader's behalf than parsing the former.\footnote{The unseemly definition of the function $h'$ will  hopefully pique the reader's interest in \cite{erdos1950combinatorial}.} 
	The proof in each of the two cases begins by defining a partition $h\colon [\bbN]^3\to \{0,1\}$ by 
	\[
	h(\{l,m,n\}_<)=
	\begin{cases}
		0, & \text{ if }f(x_l{}^\frown y_l)=f(x_m{}^\frown y_n),\\
		1, & \text{ if } f(x_l{}^\frown y_l)\neq f(x_m{}^\frown y_n). 
	\end{cases}
	\]
	By Ramsey's theorem, there is an infinite $h$-homogeneous $H\subseteq \bbN$. We claim that~$H$ cannot be 0-homogeneous. Assume otherwise. Then 0-homogeneity of the set $H$ implies that for $l<m<n<k$ in $H$ we have 
	\[
	f(x_m{}^\frown y_m)=f(x_n{}^\frown y_k)=f(x_l{}^\frown y_l)=f(x_m{}^\frown y_n), 
	\]
	contradicting  \eqref{5.Dependence}.
	Also, for $l<m<n<k$ in $H$ we have
	\[
	f(x_m,y_n)=f(x_l,y_l)=f(x_n,y_k)=f(x_m,y_m),
	\]
	contradicting \eqref{4.Dependence}.  
		
	Therefore in each of the cases \eqref{5.Dependence} and \eqref{4.Dependence} the set $H$ is $1$-homogeneous and for all $l<m<n$ we have $f(x_l{}^\frown y_l)\neq f(x_m{}^\frown y_n)$. By Ramsey's Theorem (and passing to subsequences of $x_m$ and $y_m$) we may assume that either (i) for all $m<n$ we have  $f(x_m{}^\frown y_m)=f(x_n{}^\frown y_n)$ or that (ii) for all $m<n$ we have  $f(x_m{}^\frown y_m)\neq f(x_n{}^\frown y_n)$. 
	
	If (i) applies, then for all $m$ and $k<l$ we have $f(x_m{}^\frown y_m)=f(x_k{}^\frown y_k)\neq f(x_k{}^\frown y_l)$, and \eqref{2.Dependence} holds. 
	If (ii) applies, define $h'\colon [\bbN]^3\to \{0,1\}$ by 
	\[
	h'(\{l,m,n\}_<)=\begin{cases}  
		1, &\text{ if } f(x_l{}^\frown y_l)=f(x_m{}^\frown y_n),\\
		2, &\text{ if } f(x_l{}^\frown y_l)=f(x_l{}^\frown y_m),\\
		3, &\text{ if } f(x_m{}^\frown y_m)=f(x_l{}^\frown y_m),\\
		4, &\text{ if } f(x_m{}^\frown y_m)=f(x_l{}^\frown y_n),\\
		5, &\text{ if } f(x_n{}^\frown y_n)=f(x_l{}^\frown y_m),\\
		0, & \text{ if none of the above cases applies}. 
	\end{cases}
	\]
	Since all functions $m\mapsto f(x_m{}^\frown y_m)$  and $n\mapsto f(x_m{}^\frown x_n)$ for $n>m$ are injective, $h'$ admits no infinite $j$-homogeneous sets for any $j\geq 1$. An infinite $0$-homogeneous set gives family as in \eqref{2.Dependence}. This concludes the proof. 
\end{proof}

In the proof of Theorem~\ref{T.Dependence} we use the consequence of the Axiom of Choice, compactness of the space $\beta X$ of ultrafilters on $X$ (i.e., its  \v Cech--Stone compactification) is compact. If $\cU $ is an ultrafilter on $X$ and $\varphi(x)$ is a formula, then 
\[
(\cU x) \varphi(x)
\]
stands for `the set of $x\in X$ for which $\varphi(x)$ holds belongs to $\cU$'. The quantifier $(\cU x)$ is self-dual, meaning that all $\cU$ and $\varphi$ satisfy
\[
\neg (\cU x)\varphi(x)\Leftrightarrow (\cU x)\neg \varphi(x), 
\]
as is easy to check. 

\begin{lemma} \label{L.Dependence.UF} If $d\geq 1$, $\cU_i$ for $i<d$ is an ultrafilter on $X$, and $f\colon X^d\to Y$ then either \eqref{2.Dependence} of Theorem~\ref{T.Dependence} applies or 
	\begin{enumerate}
		\popcounter
		\item \label{3.Dependence} there are $A_i\in \cU_i$ for $i<d$, $j<d$, and $g\colon A_j\to Y$ such that $f$ and $g\circ \pi_j$ agree on $\prod_{i<d} A_i$.
		\pushcounter
	\end{enumerate} 
\end{lemma}

\begin{proof}
	By induction on $d$. The case $d=1$ is trivial (take $j=0$, $A_0=X$,  and $g=f$). We now prove the case when $d=2$ because it will be used in the proof of the inductive step. 
	
	Fix $f\colon X^2\to Y$ and ultrafilters $\cU$ and $\cV$ on $X$. If at least one of $\cU$ and $\cV$ is principal, then the assertion follows from the case $d=1$. We can therefore assume that both $\cU$ and $\cV$ are nonprincipal. The proof splits into three cases. 
	
	Assume for a moment that $(\cU x) (\exists g(x)\in Y) (\cV y)f(x,y)=g(x)$. 
	If in addition there is $A_1\in \cV$ such that  $A_0=\{x\mid  (\exists g(x)\in Y) (\cV y)f(x,y)=g(x)\}$ satisfies
	$(\forall x\in A_0) (\exists g(x)\in Y) (\forall y\in A_1)f(x,y)=g(x)$, then  $A_0,A_1$, $j=0$,  and $g$ witness that \eqref{3.Dependence} holds. 
	Otherwise, there is $B\in \cU$ such that for all $x\in B$ and every $A_1\in \cV$ some $y=y(x,A_1)\in A_1$ satisfies $f(x,y)\neq g(x)$. Let $u=\{0\}$ and  $v=\{1\}$. We will find $x_m\in X^u$, $y_m\in X^v$, and $C_m\in \cV$ such that  all $m\geq 0$ satisfy 
	\begin{enumerate}[label = (\roman *)]
		\item $C_m\supseteq C_{m+1}$, 
		\item For all $y\in C_m$ we have $g(x_m)=f(x_m,y)$ and and $g(x_m)\neq f(x_m,y_m)$. 
		\item $y_{m+1}\in C_m$. 
	\end{enumerate}
	Take $x_0\in B$, let $C_0=\{y\in X\mid f(x,y)=g(x)\}$, and let  $y_0=y(x_0,C_0)$ so that $f(x_0,y_0)\neq g(x_0)$. If $x_n,y_n,C_n$, had been chosen for $n<m$ and satisfy the conditions, then 
	take $x_m\in B\setminus \{x_n\mid n<m\}$, let  
	\[
	C_m=\{y\in C_{m-1}\mid f(x,y)=g(x)\}
	\] 
	and $y_m=y(x_m,C_m)$, so that $f(x_m,y_m)\neq g(x_m)$.  This describes the construction.

	For all 
	$l<m<n$ we have 
	$	f(x_l,y_l)\neq f(x_l,y_m)$ and $f(x_l,  y_m)=f(x_l, y_n)$
	and by Lemma~\ref{L.Dependence.2}, \eqref{2.Dependence} of Theorem~\ref{T.Dependence} follows. 
	
	This concludes the discussion of the first case.

	The second case is when $(\cV y) (\exists g(y)\in Y) (\cU y)f(x,y)=g(y)$. Then the function $f'(x,y)=f(y,x)$ satisfies the assumptions of the first case, and the conclusion follows.

	We may therefore assume that neither of the first two cases applies, hence 
	\begin{align*}
		(\cU x)(\forall c) (\cV y) f(x,y)&\neq c,\\
		(\cV y)(\forall c) (\cU x) f(x,y)&\neq c. 
	\end{align*}
	Let $u=\{0\}$, $v=\{1\}$. We will find $x_m\in X^u$ and $y_m\in X^v$ such that  \eqref{4.Dependence} of Lemma~\ref{L.Dependence.2} holds, hence all $0\leq m<n$ satisfy
		\begin{enumerate}\popcounter
		\item \label{4+.Dependence} $f(x_m,y_m)\neq f(x_m,y_n)$  and $f(x_m,y_m)\neq f(x_n,y_m)$. 
		\pushcounter
	\end{enumerate}
	Towards this, we will also find $B_m\in \cU$, $C_m\in \cV$, $x_m\in B_m$, and $y_m\in C_m$ for $m\geq 0$  such that the following conditions hold for all $m$. 
	\begin{align*}
		(\forall x\in B_{m+1}) f(x,y_m)&\neq f(x_m,y_m),\\
		(\forall y\in C_{m+1}) f(x_m,y)&\neq f(x_m,y_m),\\
		C_{m+1}\subseteq C_m,\qquad &B_{m+1}\subseteq B_m. 
	\end{align*}
	Let $B_0=\{x\mid (\forall c) (\cV y) f(x,y)\neq c\}$ and 
	$C_0=\{y\mid (\forall c) (\cU x) f(x,y)\neq c\}$. These sets belong to $\cU$ and $\cV$, respectively,  Choose $x_0\in B_0$ and $y_0\in C_0$. Then  $B_1=\{x\in B_0\mid f(x,y_0)\neq f(x_0,y_0)\}$ belongs to $\cU$ and $C_1=\{y\in C_0\mid f(x_0,y)\neq f(x_0,y_0)\}$ belongs to $\cV$. If $x_m,y_m,B_m$, and $C_m$ had been chosen, pick $x_{m+1}\in  B_m$ and $y_{m+1}\in C_m$. Then the sets
	\begin{align*}
B_{m+1}&=\{x\in B_m\mid f(x,y_{m+1})\neq f(x_{m+1},y_{m+1})\}\\
C_{m+1}&=\{y\in C_m\mid f(x_{m+1},y)\neq f(x_{m+1},y_{m+1})\} 
	\end{align*}  belong to $\cU$ and  $\cV$, respectively. This describes the construction of sequences that satisfy \eqref{4+.Dependence}. 
	
	Lemma~\ref{L.Dependence.2} implies that \eqref{2.Dependence} of Theorem~\ref{T.Dependence} applies. 
	This completes the proof of Lemma~\ref{L.Dependence.UF} in case when $d=2$. 
	
	Assume that the conclusion holds for $d\geq 2$. Fix $f\colon \prod_{i<d+1} X\to Y$ and ultrafilters $\cU_i$, for $i<d+1$, on $X$. We may assume that each $\cU_i$ is nonprincipal.  For $x\in X$ let 
	\(
	f^x\colon \prod_{i<d} X\to Y
	\)
	be defined by 
	\[
	f^x(x_0,\dots, x_{d-1})=f(x_0,\dots, x_{d-1}, x).
	\] 
	By the inductive hypothesis, for every $x\in X$ the function $f^x$ satisfies one of the alternatives given by  Theorem~\ref{T.Dependence}. If for some $x\in X$ the function $f^x$ satisfies  \ref{2.Dependence} with $d=\bar u\sqcup \bar v$,$\bar x_m\in X^{\bar u}$, $\bar y_m\in X^{\bar v}$ then let $u=\bar u\cup \{d\}$, $v=\bar v$, $x_m=\bar x_m{}^\frown x$, and $y_m=\bar y_m$. These objects witness that $f$ satisfies \eqref{2.Dependence}.

	We may therefore assume that \eqref{1.Dependence} of Theorem~\ref{T.Dependence} holds for $f^x$ for all $x\in X$. For each $x\in X$ fix $A^x_i\in \cU_i$ for $i<d$, $j(x)<d$, and $g^x\colon X\to Y$ such that
	\begin{equation*}\label{eq.Dependence.gx}
		f^x(y)=(g^x\circ \pi_{j(x)})(y)\text{ for all }y\in \prod_{i<d} A^x_i. 
	\end{equation*}
	Let $j<d$ be such that $(\cU_d x) j(x)=j$. As in the case when $d=2$, we consider cases. 
	
	Assume for a moment that there are $A_i\in \cU_i$ for $i<d+1$ such  that for every $x\in A_{d}$   the function $g^x$ is defined on $A_j$ and $f^x$ and $g^x\circ \pi_j$ agree on $\prod_{i<d} A_i$. 
	Let $\bar f(x,y)=g^x(y)$. Denoting the projection from $\prod_{i<d+1} A_i$ to $A_j\times A_d$ by $\pi_{j,d}$, the restriction of $f$ to $\prod_{i<d+1} A_i$ agrees with $\bar f\circ \pi_{j,d}$. By the already proven case $d=2$, the function $\bar f$ satisfies \eqref{3.Dependence}  of Lemma~\ref{L.Dependence.UF} and it is straightforward to see that in each of the two cases this conclusion carries to $f$. 
	
	We may therefore assume that for all $\bar A=(A_i)_{i<d}$ in $\prod_{i<d}\cU_i$  there is $B(\bar A)\in \cU_d$ such that for every $x\in B(\bar A)$ some $z=z(x,\bar A)$ in $A_j$ and $y=y(x,\bar A)$ in $\prod_{i<d, i\neq j}A_i$ satisfy
	\begin{equation}
		\label{eq.Dependence}\textstyle
		f^x(y)\neq (g^x\circ \pi_{j(x)})(y). 
	\end{equation}
	Let $u=\{j,d\}$ and $v=(d+1)\setminus u$. 
	We will proceed to choose sequences $x_m\in X^u$, $y_m\in X^v$,  and $\bar A_m=(A_{i,m})_{i<d}$ in $\prod_{i<d} \cU_i$ so that these objects satisfy the following  (it will be convenient to present $x_m$ as $x_{m}(j)^\frown x_{m}(d)$). 
	\begin{enumerate}[label = (\roman *)]
		\item $A_{i,0}=X$ for $i<d$. 
		\item $x_0(d)\in B(\bar A_0)$, $x_0(j)=z(x_0(d), \bar A_0)$, and $y_0=y(x,\bar A_0)$. 
		\item \label{3+.Dependence} $A_{i,m+1}=A^{x_m(0)}_i\cap A_{i,m}$ for all $i<d$. 
		\item $x_{m+1}(d)\in B (\bar A_{m+1})$, $x_{m+1}(j)=z(x_{m+1}(d), \bar A_{m+1})$. 
		\item \label{5+.Dependence} $y_{m}=y(x_m, \bar A_m)$. 
	\end{enumerate}
	Clause \ref{3+.Dependence} implies 
	$A_{i,m+1}\in \cU_i$  for  all $i<d$ and $f^{x_{m}}(y)=g^{x_m(d)}(x_m(j))$ for all $y\in \prod_{i\in v} A_{i,m+1}$.	 
	As \ref{5+.Dependence} implies that all $y\in \prod_{i<d,i\neq j} A_{i,m-1}$ satisfy $f(x_m{}^\frown y_m)\neq g^{x_m(d)}(x_m(j))$, we have $f(x_m{}^\frown y_m)\neq f(x_m{}^\frown y_n )$ and 
	$f(x_m{}^\frown y_n)= f(x_m{}^\frown y_k )$ for all $m<n<k$. Lemma~\ref{L.Dependence.2} \eqref{5.Dependence} implies that \eqref{2.Dependence} of Theorem~\ref{T.Dependence} holds and completes the proof. 
\end{proof}	

\begin{proof}
	[Proof of Theorem~\ref{T.Dependence}] 
	Fix $f\colon \prod_{i<d} X_i\to Y$. Suppose for a moment that for every $\bar\cU=(\cU_i)_{i<d}$  in $(\beta X)^d$ for $i<d$ there are sets $A_i^{\bar \cU}\in \cU_i$, for $i<d$, $j<d$, and $g^{\bar \cU}\colon U_j\to Y$  such that the restriction of $f$ to $\prod_{i<d} A_i^{\bar \cU}$ agrees with $g\circ \pi_j$. These sets correspond to an open neighbourhood of $\bar \cU$ hence compactness of $(\beta X)^d$ implies \ref{1.Dependence} of Theorem~\ref{T.Dependence}. 
	
	We may therefore assume that for some $\bar \cU$ there are no such objects, and Lemma~\ref{L.Dependence.UF} implies that \eqref{2.Dependence} of Theorem~\ref{T.Dependence} holds. 	
\end{proof}

\section{Concluding remarks}
\label{S.Choice} 

We do not know what fragment of the Axiom of Choice  is needed to prove Theorem~\ref{T.Dependence}, but we do know that it cannot be proven in ZF. Recall that a set is \emph{Dedekind-finite} if it admits no injection into a proper subset of itself. It is well-known that ZF is relatively consistent with the existence of an infinite, Dedekind--finite set (see e.g., \cite{jech2008axiom}).

\begin{proposition} Suppose that $X$ is a Dedekind-finite, infinite set. 
	Then there exists $f\colon X^2\to \{0,1\}$ such that both alternatives of Theorem~\ref{T.Dependence} fail.
	\end{proposition}
	
	\begin{proof}
		Let $f(x,y)=0$ if $x=y$ and $f(x,y)=1$ if $x\neq y$. If $X=\bigcup_{i<k} U_i$ then since $X$ is infinite there exists $j<k$ such that $X^2\cap U_j^2$ is infinite. Clearly  the function $f$ depends on both variables on $U_j^2$. 
		
		Assume that \eqref{2.Dependence} holds. Then each one of $u$ and $v$ is a singleton and   $x_m$, $y_m$, for $m\in \bbN$,  are sequences of elements of $X$.  Since $X$ is Dedekind-finite, there are $m<n$ such that $y_m=y_n$. Then $f(x_m{}^\frown y_m)=f(x_m{}^\frown y_n)$; contradiction. 
	\end{proof}

	The reader may object that in choiceless context a more natural analog of~\eqref{2.Dependence} should involve sequences $(x_m)$ and $(y_m)$ indexed by some infinite (possibly Dedekind-finite) sets instead of $\bbN$. Clearly if $X$ is Dedekind-finite then the partition into singletons is Dedekind-finite and each of the rectangles is a singleton, hence $f$ is constant on it and allowing this version of (1) leads nowhere. We may therefore consider the variant in which \eqref{1.Dependence} of Theorem~\ref{T.Dependence} is unchanged but \eqref{2.Dependence}  is modified by allowing the sequences to be indexed by a Dedekind-finite set.  A stronger assumption refutes this alternative. A set $X$ is a \emph{Russell set} if $X=\bigsqcup_{n\in \bbN} X_n$ where each $X_n$ has two elements but for every $Y\subseteq X$ the set $\{n\mid |Y\cap X_n|=1\}$ is finite.  A Russell set is Dedekind-finite and it is well-known that if ZF is relatively consistent with the existence of  Russell set (see e.g., \cite{jech2008axiom}). 
	
	\begin{proposition}If $X$ is a Russell set then there is $f\colon X^2\to \{0,1\}$ such that for every $k\in \bbN$ and partition $X=\bigsqcup_{i<k} U_i$ for some $i<k$ the restriction of $f$ to $U_i^2$ depends on both coordinates and for every infinite set $Z$ and all $x_z,y_z$ in $X$ there are arbitrarily large finite sets $I\subseteq Z$ such that the restriction of $f$ to $\{x_z\vert z\in Z\}\times \{y_z\vert z\in Z\}$ is constant. 
	\end{proposition}

	\begin{proof}
		We have that $X=\bigsqcup_{n\in \bbN} X_n$,  each $X_n$ has two elements and for every $Y\subseteq X$ the set $\{n\mid |Y\cap X_n|=1\}$ is finite. Let $f(x,y)=0$ if $\{x,y\}=X_n$ for some $n$ and $f(x,y)=1$ otherwise. 
		If $X=\bigsqcup_{i<k} U_i$, then for some $i$ the set $\{n\vert X_n\cap U_i\neq \emptyset\}$ is infinite. Thus the set 
		$\{n\vert X_n\subseteq U_i\}$ is infinite, and the restriction of $f$ to $U_i^2$ is isomorphic to the restriction of $f$ to $X^2$.  

	Now assume that $Z$ is an infinite set and that $x_z, y_z$, for $z\in Z$ belong to $X$. Then the set of $n$ such that some $z$ satisfies $\{x_z,y_z\}=X_n$ is finite, therefore all but finitely many $z$  satisfy $f(x_z,y_z)=1$. 
		For every $x\in X$ we have that $f(x,y)=0$ or $f(y,z)=0$ for exactly one $y\in X$. One can therefore for every $k\geq 1$ recursively choose $F\subseteq Z$ of cardinality $k$ such that $f(x_z,y_{z'})=1$ for all $z$ and $z'$ in $F$. 
	\end{proof}
	
	We conclude with an easy self-strengthening of Theorem~\ref{T.Dependence}.

	\begin{proposition}[ZF] For all $d\geq 1$, well-orderable sets $X_i$, for $i<d$ and $Y$, every function $f\colon \prod_{i<d} X_i\to Y$  
		satisfies one of the alternatives of Theorem~\ref{T.Dependence}. 
		\end{proposition}
		
		\begin{proof}
			As in Theorem~\ref{T.Dependence},  it suffices to consider the case when $f\colon X^d\to Y$. 
			Clearly $X^d$ is well-orderable, and so is the range of $f$, as an image of the well-orderable set $X^d$. 
			Therefore we can identify $X$ and $Y$ with ordinals, and the graph of $f$ with a subset of $X^d\times Y$. The model $L[f]$  is a model of ZFC (\cite[Exercise~II.6.30]{Ku:Set}). 
			Each of the alternatives of Theorem~\ref{T.Dependence} is a $\Sigma_1$ statement and therefore holds in $V$. 
		\end{proof}


In \cite[Question~11]{Fa:Dependence} I asked whether there is a finitary version of Theorem~\ref{T.Dependence}, without suggesting what such finitary version should look like. A more precise question (to which I do not know the answer) is whether for all $d$ and $k$ there is $N=N(d,k)$ such that for all $X$ and all $f\colon X^d\to X$, if there are no $d=u\sqcup v$, $x_m\in X^u$, $y_m\in Y^v$  for $m<n$ such that 
		$f(x_l{}^\frown y_l)\neq f(x_m{}^\frown y_n)$ for all $ l<N$ and $m<n<N$, then there is a partition $X=\bigsqcup_{j<k} U_{i}$ such that for every $s\in k^d$ there are $j(s)<d$  and $g_s\colon U_{j(s)}\to Y$ such that $f$ agrees with $g_s\circ \pi_{j(s)}$ on $\prod_{i<d} U_{s(i)}$. 
		
		Finally a word on the motivation for this work. Theorem~\ref{T.Dependence}  was inspired by van Douwen's \cite{vD:Prime}. By using ideas from \cite{vD:Prime} and \cite{Fa:AQ}, this theorem was used \cite{Fa:Dimension} and \cite{Fa:Powers} to prove that \v Cech--Stone remainders of 0-dimensional spaces exhibit some `dimension phenomena'.  For example, $(\bbN^*)^n$ admits a continuous surjection onto $(\bbN^*)^{n+1}$ for some $n\geq 1$ if and only if this holds for $n=1$ (\cite{Fa:Powers}). This strengthens the main result of \cite{Just:Omega^n} modulo the available lifting results. See \cite{yilmaz2023weak} for the ultimate result along these lines. 
		
	\bibliography{ifmainbib}
	\bibliographystyle{plain}
	
\end{document}